\documentclass[a4paper,10pt]{amsart}

\usepackage{amsmath}
\usepackage{amsthm}
\usepackage{amssymb}
\usepackage{amsfonts}
\usepackage{enumitem} 
\usepackage[bookmarks=true,hyperindex,pdftex,colorlinks,citecolor=blue]{hyperref}

\DeclareMathOperator{\lspan}{span}                          
\DeclareMathOperator{\supp}{supp}                           
\DeclareMathOperator{\rad}{rad}                             
\DeclareMathOperator{\Lip}{Lip}                             

\newcommand{\NN}{\mathbb{N}}             
\newcommand{\ZZ}{\mathbb{Z}}             
\newcommand{\RR}{\mathbb{R}}             

\newcommand{\abs}[1]{\left|{#1}\right|}                     
\newcommand{\pare}[1]{\left({#1}\right)}                    
\newcommand{\set}[1]{\left\{{#1}\right\}}                   
\newcommand{\norm}[1]{\left\|{#1}\right\|}                  
\newcommand{\dual}[1]{{#1}^\ast}                            
\newcommand{\ddual}[1]{{#1}^{\ast\ast}}                     
\newcommand{\ball}[1]{B_{{#1}}}                             
\newcommand{\duality}[1]{\left<{#1}\right>}                 
\newcommand{\cl}[1]{\overline{#1}}                          
\newcommand{\restrict}{\mathord{\upharpoonright}}           

\newcommand{\lipfree}[1]{\mathcal{F}({#1})}                 
\newcommand{\lipnorm}[1]{\norm{#1}_L}                       
\newcommand{\pos}[1]{#1^+}                                  
\newcommand{\wop}[1]{T_{#1}}                                
\newcommand{\bidualfree}[1]{\dual{\Lip_0({#1})}}            
\newcommand{\biduallip}[1]{\ddual{\Lip_0({#1})}}            

\theoremstyle{plain}
\newtheorem{theorem}{Theorem}
\newtheorem{lemma}[theorem]{Lemma}

\newtheorem{claim}{Claim}
\newtheorem*{claim*}{Claim}

\theoremstyle{definition}
\newtheorem*{definition*}{Definition}
\newtheorem{definition}[theorem]{Definition}

\theoremstyle{remark}

\begin{document}

\title{Normal functionals on Lipschitz spaces are weak$^\ast$ continuous}

\author[R. J. Aliaga]{Ram\'on J. Aliaga}
\address[R. J. Aliaga]{Instituto Universitario de Matem\'atica Pura y Aplicada, Universitat Polit\`ecnica de Val\`encia, Camino de Vera S/N, 46022 Valencia, Spain}
\email{raalva@upvnet.upv.es}

\author[E. Perneck\'a]{Eva Perneck\'a}
\address[E. Perneck\'a]{Faculty of Information Technology, Czech Technical University in Prague, Th\'akurova 9, 160 00, Prague 6, Czech Republic}
\email{perneeva@fit.cvut.cz}

\date{} 


\begin{abstract}
Let $\mathrm{Lip}_0(M)$ be the space of Lipschitz functions on a complete metric space $M$ that vanish at a base point. We prove that every normal functional in ${\mathrm{Lip}_0(M)}^*$ is weak$^*$ continuous, i.e. in order to verify weak$^*$ continuity it suffices to do so for bounded monotone nets of Lipschitz functions. This solves a problem posed by N. Weaver. As an auxiliary result, we show that the series decomposition developed by N. J. Kalton for functionals in the predual of $\mathrm{Lip}_0(M)$ can be partially extended to ${\mathrm{Lip}_0(M)}^*$.

\


\


\end{abstract}

\maketitle


\section{Introduction}

Let $(M,d)$ be a complete metric space with a selected base point, which we shall denote by $0$. Then the space $\Lip_0(M)$ of all real-valued Lipschitz functions on $M$ that vanish at $0$ is a Banach space when endowed with the norm given by the Lipschitz constant
$$
\lipnorm{f}=\sup\set{\frac{\abs{f(x)-f(y)}}{d(x,y)}:x\neq y\in M}
$$
(the requirement that $f(0)=0$ gets rid of the 
constant functions, otherwise $\lipnorm{\cdot}$ is merely a seminorm).
Moreover, $\Lip_0(M)$ is a dual Banach space. Its canonical predual $\lipfree{M}$, usually called \emph{Lipschitz-free space} or \emph{Arens-Eells space} over $M$, can be realized as the subspace $\lipfree{M}=\cl{\lspan}\set{\delta(x):x\in M}$ of $\dual{\Lip_0(M)}$, where $\delta(x)\in\dual{\Lip_0(M)}$ denotes the evaluation functional on $x\in M$. Note that $\delta$ is an isometric embedding of $M$ into $\dual{\Lip_0(M)}$, so $\lipfree{M}$ contains a linearly dense and linearly independent isometric copy of $M$.

The Lipschitz spaces $\Lip_0(M)$ are in many ways the metric counterparts of the classical $C(K)$ spaces of real-valued continuous functions on Hausdorff compacts, so their study is interesting in its own right; for a detailed analysis of their properties, see the reference monograph \cite{Weaver2} by Weaver. However, they currently attract a lot of attention due to their applications to the nonlinear geometry of Banach spaces. These usually involve the following extension property satisfied by Lipschitz-free spaces: \textit{any Lipschitz mapping from $M$ into a Banach space $X$ can be extended to a linear operator from $\lipfree{M}$ into $X$ whose norm is the Lipschitz constant of the original mapping} (here, each $x\in M$ is identified with its associated evaluation functional $\delta(x)\in\lipfree{M}$). In \cite{GoKa_2003}, Godefroy and Kalton famously used this to prove that the bounded approximation property of Banach spaces is stable under Lipschitz isomorphisms. Since then, numerous other applications to nonlinear functional analysis have been found; see e.g. the recent survey \cite{Godefroy_2015} by Godefroy.

The weak$^\ast$ topology induced by $\lipfree{M}$ on $\Lip_0(M)$ coincides with the topology of pointwise convergence on norm-bounded subsets of $\Lip_0(M)$. Therefore, by a straightforward application of the Banach-Dieudonn\'e theorem, a functional  $\phi\in\bidualfree{M}$ is weak$^\ast$ continuous (i.e. it belongs to $\lipfree{M}$) precisely when it satisfies the following condition: given any norm-bounded net $(f_i)$ in $\Lip_0(M)$ that converges pointwise to $f\in\Lip_0(M)$, one has that $\duality{f_i,\phi}$ converges to $\duality{f,\phi}$.

In \cite{Weaver}, Weaver considered the following weaker notion, by analogy with the corresponding notion for von Neumann algebras:
\begin{definition}
A functional $\phi\in\bidualfree{M}$ is \emph{normal} when it satisfies the following: given any norm-bounded net $(f_i)$ in $\Lip_0(M)$ that converges pointwise \emph{and monotonically} to $f\in\Lip_0(M)$, one has that $\duality{f_i,\phi}$ converges to $\duality{f,\phi}$.
\end{definition}
\noindent Equivalently, $\phi$ is normal if $\duality{f_i,\phi}\rightarrow 0$ for any net $(f_i)$ of non-negative functions in $\ball{\Lip_0(M)}$ that decreases pointwise to $0$.

By a well-known theorem, states on a von Neumann algebra are normal if and only if they belong to its predual (see e.g. \cite[Theorem 1.13.2]{Sakai}). In particular, since normality only depends on the order structure of the von Neumann algebra, this implies that von Neumann algebras have unique preduals \cite[Corollary 1.13.3]{Sakai}. In our setting, any weak$^\ast$ continuous element of $\bidualfree{M}$ is obviously normal.
Weaver asked in \cite[Open problem on p. 37]{Weaver} whether the converse is also true. He first gave an affirmative answer for the very specific case of evaluation functionals on elements of the Stone-\v Cech compactification of $M$ \cite[Proposition 2.1.6]{Weaver} and for weak$^\ast$ limits of nets of elementary molecules \cite[Theorem 3.43]{Weaver2}. Later, he extended the result to all positive functionals \cite[Theorem 2.3]{Weaver_2018}, i.e. those $\phi\in\bidualfree{M}$ such that $\duality{f,\phi}\geq 0$ for any non-negative $f\in\Lip_0(M)$. This allowed him to show, similarly to von Neumann algebras, that the Lipschitz-free space $\lipfree{M}$ is in fact the unique predual of $\Lip_0(M)$ when $M$ is bounded or geodesic \cite{Weaver_2018}. It is currently an open problem whether this holds for all metric spaces $M$.

In this short note, we settle the question about normality in the general case:

\begin{theorem}
\label{th:normal}
Let $M$ be a complete pointed metric space and $\phi\in\bidualfree{M}$. Then $\phi$ is normal if and only if it is weak$^\ast$ continuous.
\end{theorem}

\noindent Let us note that Theorem \ref{th:normal}, besides being an analog of the corresponding von Neumann algebra result, can also be considered as an abstract version of the Radon-Nikod\'ym theorem for Lipschitz-free spaces; compare e.g. to \cite[Theorem 8.7]{Schaefer}. The classical Radon-Nikod\'ym theorem implies that $L_1$ is $1$-complemented in its bidual $L_1^{\ast\ast}$ (see e.g. \cite[Proposition 6.3.10]{AK}). Therefore, since $\lipfree{\RR}$ is isometric to $L_1(\RR)$, we obtain that $\lipfree{\RR}$ is complemented in its bidual $\bidualfree{\RR}$. In a deep paper \cite{CKK}, C\'uth, Kalenda and Kaplick\'y extended this result and proved that the Lipschitz-free space over any finite-dimensional Banach space is complemented in its bidual. This is however not true in general in the infinite-dimensional case; for instance, $\lipfree{c_0}$ is not complemented in its bidual as it contains a complemented copy of $c_0$ by the lifting property \cite[Theorem 3.1]{GoKa_2003}. It remains an important open problem to decide for which metric spaces $M$ the Lipschitz-free space $\lipfree{M}$ is complemented in its bidual. Of particular interest is the case when $M=\ell_1$, because the complementability would imply that $\ell_1$ is determined by its Lipschitz structure (see e.g. \cite[Problem 16]{GLZ_2014}). Based on the similarity to the Radon-Nikod\'ym theorem, one might try to investigate whether Theorem \ref{th:normal} could be helpful in addressing this problem.

In order to give the proof of Theorem \ref{th:normal} in Section 3, we first establish some auxiliary results concerning series decomposition of functionals on Lipschitz spaces in Section 2.

Let us now briefly introduce the notation used in this note. $B_X$ will stand for the closed unit ball of a Banach space $X$. The closed ball with radius $r$ around $x\in M$ will be denoted $B(x,r)$. We will use the notation
\begin{align*}
d(x,A) &= \inf\set{d(x,a):a\in A} \\
\rad(A) &= \sup\set{d(0,a):a\in A}
\end{align*}
for $x\in M$ and $A\subset M$. $\Lip_0(M)^+$ will be the set of all non-negative functions in $\Lip_0(M)$. The pointwise maximum and minimum of real-valued functions $f$ and $g$ will be written as $f\vee g$ and $f\wedge g$, respectively. We will also denote $f^+=f\vee 0$ and $f^-=(-f)\vee 0$. Note that $f=f^+-f^-$. By the support of $f$ we mean the set 
$$
\supp(f)=\set{x\in M: f(x)\neq 0}
$$
and we will put $\norm{f}_{\infty}=\sup\set{\abs{f(x)}: x\in M}$, which can be infinite.

Let us recall that for any two Lipschitz functions $f,g$ on $M$ we have
$$
\lipnorm{fg}\leq\lipnorm{f}\norm{g}_\infty+\lipnorm{g}\norm{f}_\infty .
$$
It follows that for any Lipschitz function $h$ on $M$ with bounded support, the mapping 
$$\wop{h}\colon f\mapsto f\cdot h$$
is a linear operator on $\Lip_0(M)$ whose norm is bounded by
\begin{equation}
\label{eq:weighting_op_norm}
\norm{\wop{h}} \leq \norm{h}_\infty + \rad(\supp(h))\lipnorm{h} .
\end{equation}
Moreover $\wop{h}$ is weak$^\ast$-weak$^\ast$-continuous, i.e. its adjoint $\dual{\wop{h}}:\phi\rightarrow\phi\circ\wop{h}$ takes $\lipfree{M}$ into $\lipfree{M}$. See \cite[Lemma 2.3]{APPP_2020} for the proof of these facts. We will be using these operators with weighting functions $h$ such that $h=1$ on some region of interest $A\subset M$ and $h=0$ on some region $B\subset M$ which is to be ignored, and takes intermediate values in some transition region. In particular, we will consider the functions $\Lambda_n$ for $n\in\ZZ$ defined by
$$
\Lambda_n(x) = \begin{cases}
0 &\text{, if } d(x,0)\leq 2^{n-1} \\
2^{-(n-1)}d(x,0)-1 &\text{, if } 2^{n-1}\leq d(x,0)\leq 2^n \\
2-2^{-n}d(x,0) &\text{, if } 2^n\leq d(x,0)\leq 2^{n+1} \\
0 &\text{, if } 2^{n+1}\leq d(x,0)
\end{cases}
$$
and $\Pi_n$ for $n\in\NN$, defined by
$$
\Pi_n(x) = \begin{cases}
0 &\text{, if } d(x,0)\leq 2^{-(n+1)} \\
2^{n+1}d(x,0)-1 &\text{, if } 2^{-(n+1)}\leq d(x,0)\leq 2^{-n} \\
1 &\text{, if } 2^{-n}\leq d(x,0)\leq 2^n \\
2-2^{-n}d(x,0) &\text{, if } 2^n\leq d(x,0)\leq 2^{n+1} \\
0 &\text{, if } 2^{n+1}\leq d(x,0)
\end{cases}
$$
for $x\in M$. Notice that
\begin{equation}
\label{eq:Pi_n_Lambda_n}
\Pi_n=\sum_{k=-n}^n\Lambda_k
\end{equation}
for any $n\in\NN$. Moreover, $\norm{\Lambda_k}_{\infty},\norm{\Pi_n}_{\infty}\leq 1$, and we have
\begin{equation}
\label{eq:norm_Lambda_Pi}
\begin{gathered}
\rad(\supp(\Lambda_k))\leq 2^{k+1},\quad \lipnorm{\Lambda_k}\leq 2^{-(k-1)}\\
\rad(\supp(\Pi_n))\leq 2^{n+1},\quad\lipnorm{\Pi_n}\leq 2^{n+1}
\end{gathered}
\end{equation}
for every $k\in \ZZ$ and $n\in\NN$. In particular, \eqref{eq:weighting_op_norm} yields $\norm{\wop{\Lambda_k}}\leq 5$.

\section{Series decomposition in \texorpdfstring{$\bidualfree{M}$}{Lip0(M)*}}

In Section 4 of \cite{Kalton_2004}, Kalton established that elements of $\lipfree{M}$ admit a decomposition as a series with terms whose action is limited to annuli around the base point. Let us prove that this decomposition is also valid for normal functionals in $\bidualfree{M}$. We will use a slightly different version of the decomposition, based on the functions $\Lambda_n$ instead of the original ones because they make computations easier.

\begin{lemma}
\label{lm:kalton_2020}
For any $\phi\in\bidualfree{M}$ we have
\begin{equation}
\label{eq:kalton_bound}
\sum_{n\in\ZZ}\norm{\phi\circ\wop{\Lambda_n}}\leq 45\norm{\phi}.
\end{equation}
Hence
$$
\sum_{k=-n}^n\phi\circ\wop{\Lambda_k}= \phi\circ\wop{\Pi_n}
$$
converges in norm as $n\rightarrow\infty$ to a functional in $\bidualfree{M}$.
\end{lemma}

\begin{proof}
Fix $\varepsilon>0$ and a finite set $F\subset\ZZ$. For $i=0,1,2$, let $F_i$ be the set of those $n\in F$ that are congruent with $i$ modulo $3$. We will show that
$$
\sum_{n\in F_i}\norm{\phi\circ\wop{\Lambda_n}} < 15\norm{\phi}+\varepsilon
$$
and this will be enough to prove \eqref{eq:kalton_bound}. The second part of the statement is then obvious in view of \eqref{eq:Pi_n_Lambda_n}.

Fix $i$, and for $n\in F_i$ choose $f_n\in\ball{\Lip_0(M)}$ such that
$$
\norm{\phi\circ\wop{\Lambda_n}}-\frac{\varepsilon}{\abs{F_i}} < \duality{f_n,\phi\circ\wop{\Lambda_n}} = \duality{f_n\Lambda_n,\phi} .
$$
Notice that $\lipnorm{f_n\Lambda_n}\leq\norm{\wop{\Lambda_n}}\leq 5$ by \eqref{eq:weighting_op_norm} and \eqref{eq:norm_Lambda_Pi}. Now define $g=\sum_{n\in F_i}f_n\Lambda_n$ and let us estimate $\lipnorm{g}$. Fix $x\in\supp(g)$, then $x\in\supp(\Lambda_n)$ for some $n\in F_i$. If $y\in\supp(\Lambda_m)$ for $m\in F_i\setminus\set{n}$, assume $m>n$ without loss of generality, then $d(x,y)\geq d(x,0)$ and
\begin{align*}
\abs{g(x)-g(y)} &\leq \abs{f_n(x)\Lambda_n(x)}+\abs{f_m(y)\Lambda_m(y)} \\
&\leq 5(d(x,0)+d(y,0)) \\
&\leq 5(2d(x,0)+d(x,y)) \leq 15d(x,y) .
\end{align*}
Otherwise
$$
\abs{g(x)-g(y)}=\abs{f_n(x)\Lambda_n(x)-f_n(y)\Lambda_n(y)}\leq 5d(x,y) .
$$
So we get $\lipnorm{g}\leq 15$. Therefore
$$
\sum_{n\in F_i}\norm{\phi\circ\wop{\Lambda_n}} < \sum_{n\in F_i}\duality{f_n\Lambda_n,\phi} + \varepsilon = \duality{g,\phi} + \varepsilon \leq 15\norm{\phi}+\varepsilon
$$
as was claimed.
\end{proof}

\begin{lemma}
\label{lm:seq_normal_decomp}
If $\phi\in\bidualfree{M}$ is normal, then 
\begin{equation}
\label{eq:kalton_sum}
\phi=\sum_{n\in\ZZ}\phi\circ\wop{\Lambda_n}=\lim_{n\rightarrow\infty}\phi\circ\wop{\Pi_n}
\end{equation}
with respect to the norm convergence in $\bidualfree{M}$.
\end{lemma}

\begin{proof}
It will suffice to show that $(\phi\circ\wop{\Pi_n})$ converges weak$^\ast$ to $\phi$, since Lemma \ref{lm:kalton_2020} implies that the sequence converges in norm.
That is, we need to show that $\duality{f,\phi\circ\wop{\Pi_n}}\rightarrow\duality{f,\phi}$ for any $f\in\Lip_0(M)$; we may assume that $f\geq 0$, and the general case then follows by expressing $f=f^+-f^-$.

So fix $f\in\Lip_0(M)^+$. For $n\in\ZZ$ define the function $h_n$ by
$$
h_n(x) = \begin{cases}
1 &\text{, if } d(x,0)\leq 2^n \\
2-2^{-n}d(x,0) &\text{, if } 2^n\leq d(x,0)\leq 2^{n+1} \\
0 &\text{, if } 2^{n+1}\leq d(x,0)
\end{cases}
$$
for $x\in M$, which satisfies $\norm{\wop{h_n}}\leq 3$ by \eqref{eq:weighting_op_norm}. Now notice that $\Pi_n=h_n(1-h_{-(n+1)})$, hence $\wop{\Pi_n}=\wop{h_n}\circ (I-\wop{h_{-(n+1)}})$ where $I$ is the identity operator on $\Lip_0(M)$, and
$$
\norm{\wop{\Pi_n}}\leq\norm{\wop{h_n}}\norm{I-\wop{h_{-(n+1)}}}\leq 12
$$
for any $n\in\NN$. Then $\lipnorm{\wop{\Pi_n}(f)}\leq 12\lipnorm{f}$, and $\wop{\Pi_n}(f)(x)$ converges pointwise and monotonically (increasing) to $f(x)$ for every $x\in M$. By the normality of $\phi$ we have
$$
\lim_{n\rightarrow\infty} \duality{f,\phi\circ\wop{\Pi_n}} = \lim_{n\rightarrow\infty} \duality{\wop{\Pi_n}(f),\phi} = \duality{f,\phi} .
$$
This ends the proof.
\end{proof}

Moreover, each term in the decomposition series and in the limit in \eqref{eq:kalton_sum} is also normal:

\begin{lemma}
\label{lm:normal_weighted}
Let $h$ be a non-negative Lipschitz function on $M$ with bounded support. If $\phi\in\bidualfree{M}$ is normal, then $\phi\circ\wop{h}$ is normal.
\end{lemma}

\begin{proof}
Let $(f_i)$ be a bounded net in $\Lip_0(M)$ that decreases to $0$ pointwise. Then $\lipnorm{f_i h}\leq\norm{\wop{h}}\lipnorm{f_i}$ is bounded by \eqref{eq:weighting_op_norm}, so $(f_i h)$ is also a bounded net that decreases to $0$ pointwise. Since $\phi$ is normal, we have
$$
\lim_i \duality{f_i,\phi\circ\wop{h}} = \lim_i \duality{f_i h,\phi} = 0 .
$$
It follows that $\phi\circ\wop{h}$ is normal, too.
\end{proof}

\section{Proof of Theorem \ref{th:normal}}

In addition to the above decomposition result, another essential ingredient for our proof is the following simple but powerful lemma from \cite{CCGMR_2019}, which is itself based on a weaker version found in \cite{KaMaSo_2016}. We include a short proof for the sake of completeness. In fact, the same argument yields a stronger statement than the one in \cite{CCGMR_2019}. Recall that a series $\sum_nx_n$ in a Banach space $X$ is \emph{weakly unconditionally Cauchy} if $\sum_n\abs{\duality{x_n,x^\ast}}<\infty$ for every $x^\ast\in X^\ast$.

\begin{lemma}[{\cite[Lemma 1.5]{CCGMR_2019}}]
\label{lm:ccmgr}
Let $(f_n)$ be a bounded sequence in $\Lip_0(M)$. Suppose that the supports of the functions $f_n$ are pairwise disjoint. Then $\sum_n f_n$ is a weakly unconditionally Cauchy series. In particular, $(f_n)$ is weakly null.
\end{lemma}

\begin{proof}
Let $(f_n)$ be a sequence in $B_{\Lip_0(M)}$ with disjoint supports and let $(t_n)\in\ell_{\infty}$. Then 
$$\sum_{n=1}^k t_nf_n=\sum_{n=1}^k(t_nf_n)^+-\sum_{n=1}^k (t_nf_n)^-=\bigvee_{n=1}^k (t_nf_n)^+ - \bigvee_{n=1}^k (t_nf_n)^-.$$
Hence $\norm{\sum_{n=1}^k t_nf_n}_L\leq 2 \norm{(t_n)}_{\infty}$ for every $k\in\NN$ and $\sum_{n=1}^{\infty}f_n$ is weakly unconditionally Cauchy by \cite[Proposition II.D.4]{Woj}.
\end{proof}

We can now finally prove our main result.

\begin{proof}[Proof of Theorem \ref{th:normal}]

The sufficiency part of the statement is obvious. To prove the necessity, let $\phi\in\bidualfree{M}$ be a normal functional. Lemma \ref{lm:seq_normal_decomp} says that $\phi=\lim_{n\to\infty}\phi\circ\wop{\Pi_n}$ with respect to the norm convergence, so it suffices to show that $\phi\circ\wop{\Pi_n}$, for any $n\in\NN$, is weak$^\ast$ continuous. Moreover, by Lemma \ref{lm:normal_weighted}, such $\phi\circ\wop{\Pi_n}$ for any $n\in\NN$ is also normal. Therefore, for the rest of the proof we will assume that $\phi\in{\bidualfree{M}}$ is a normal functional with norm $1$, and that there exist real numbers $0<r<R$ such that $\duality{f,\phi}=0$ whenever $f\in\Lip_0(M)$ equals $0$ on the set $$K=\set{x\in M:r\leq d(x,0)\leq R}.$$

We will repeatedly make use of the function
$$e(x)=\pare{1-\frac{4}{r}d(x,K)}\vee 0 \quad\textup{for all }x\in M,$$ 
the support of which is contained in 
$$K^{\prime}=\set{x\in M:\frac{3}{4}r\leq d(x,0)\leq R+\frac{r}{4}}$$
and which equals $1$ on $K$, and the function 
$$e^{\prime}(x)=\pare{1-\frac{4}{r}d(x,K^{\prime})}\vee 0\quad\textup{for all }x\in M,$$ the support of which is contained in 
$$K^{\prime\prime}=\set{x\in M:\frac{r}{2}\leq d(x,0)\leq R+\frac{r}{2}}$$
and which equals $1$ on $K'$. (We think of them as the ``unit on $K$", which will be used to restrict functions, and the ``unit on $K^{\prime}$", which will be used to translate functions, respectively.) Note that $e,e^{\prime}\in\Lip_0(M)^+$ with $\lipnorm{e},\lipnorm{e^{\prime}}\leq\frac{4}{r}$; in particular, 
\begin{equation}
    \label{eq:phi_at_e}
    \abs{\duality{e^{\prime},\phi}}\leq \frac{4}{r}.
\end{equation}
For brevity, denote
$$
\alpha=2+(R+1)\frac{4}{r} .
$$

We will proceed by contradiction. Suppose that $\phi\notin\lipfree{M}$. By the Hahn-Banach theorem, there exists $\psi\in\ball{\biduallip{M}}$ such that $\duality{\phi,\psi}=c>0$ and that $\duality{\mu,\psi}=0$ for every $\mu\in\lipfree{M}$. Our argument relies on a construction presented in the following claim:

\begin{claim}
\label{cl:existence_one_function}
With the notation as above, for a given nonempty finite set $A\subset K^{\prime}$ and an $\varepsilon\in\left(0,\min\{1,\frac{rc}{48}\}\right)$, there exists a function $g:M\to\RR$ satisfying the following:
\begin{enumerate}[itemsep=2pt,label={\upshape{(\roman*)}}]
\item\label{g_positive Lipschitz} $g\in\Lip_0(M)^+$ with $\lipnorm{g}\leq\alpha$,
\item\label{g_max} $g(x)\leq 2\varepsilon$ for every $x\in A$,
\item\label{g_min} $g(x)\geq \varepsilon$ for every $x\in K^{\prime}$,
\item\label{g_outside support} $g(x)=\varepsilon e^{\prime}(x)$ for every $x\in M\setminus{K^{\prime}}$; in particular, $\supp(g)\subset K^{\prime\prime}$ and $\|g\|_{\infty}\leq \alpha(R+\frac{r}{2})$,
\item\label{g_big at phi} $\abs{\duality{g,\phi}}\geq\frac{c}{4}$.
\end{enumerate}
\end{claim}

\begin{proof}[Proof of Claim \ref{cl:existence_one_function}]
Consider the weak$^\ast$ neighborhood $U$ of $\psi$ in $\biduallip{M}$ given by
$$
U=\big\{\varrho\in\biduallip{M}:\abs{\duality{\phi,\varrho-\psi}}<\frac{c}{3}\textup{ and }\abs{\duality{\delta(x),\varrho}}<\varepsilon\textup{ for all }x\in A\big\}
$$
(notice that $\duality{\delta(x),\varrho-\psi}=\duality{\delta(x),\varrho}$). Thanks to the weak$^\ast$ density of $\ball{\Lip_0(M)}$ in $\ball{\biduallip{M}}$, we may find an $f\in\ball{\Lip_0(M)}\cap U$, which means that $\duality{f,\phi}>\frac{2}{3}c$ and $\abs{f(x)}=\abs{\duality{\delta(x),f}}<\varepsilon$ for every $x\in A$.
By replacing $f$ with $f^+$ or $f^-$, we obtain $f\in\ball{\Lip_0(M)}\cap\pos{\Lip_0(M)}$ such that
\begin{equation}
\label{eq:phi_at_f}
\abs{\duality{f,\phi}}>\frac{c}{3}
\end{equation}
and $f(x)<\varepsilon$ for every $x\in A$. 

Now, put 
$$g=\wop{e}(f)+\varepsilon{e^{\prime}}.$$
Then $g\in\Lip_0(M)^+$ and by \eqref{eq:weighting_op_norm} we have 
$$\lipnorm{g}\leq 1+\left(R+\frac{r}{4}\right)\frac{4}{r}+\varepsilon\frac{4}{r}\leq\alpha,$$
so $g$ satisfies \ref{g_positive Lipschitz}. Moreover, $\supp(\wop{e}(f))\subset \supp(e)\subset K^{\prime}$, which establishes \ref{g_outside support}. In particular, the bound on $\norm{g}_{\infty}$ then follows from \ref{g_positive Lipschitz} and the definition of $K^{\prime\prime}$. Properties \ref{g_max} and \ref{g_min} are straightforward to verify. Finally, since the evaluation of $\phi$ only depends on the restriction of a function to the set $K$ and since $g\restrict_K=(f+\varepsilon e^{\prime})\restrict_K$, we get by \eqref{eq:phi_at_f} and \eqref{eq:phi_at_e} that
$$\abs{\duality{g,\phi}}=\abs{\duality{f+\varepsilon e^{\prime},\phi}}\geq\abs{\duality{f,\phi}}-\varepsilon\abs{\duality{e^{\prime},\phi}}>\frac{c}{3}-\frac{c}{12}=\frac{c}{4},$$ thus \ref{g_big at phi} also holds.
\end{proof}

To proceed with the main proof, let us fix a decreasing sequence $(\varepsilon_n)_{n=1}^{\infty}\subset\left(0,\min\set{1,\frac{cr}{48},\frac{r}{2}}\right)$ such that $\varepsilon_n\rightarrow 0$ and
\begin{equation}
    \label{eq:epsilon_choice}
    (2+\alpha)\varepsilon_{n+1}<\varepsilon_n
\end{equation}
for every $n\in\NN$.

Let $\mathfrak{F}$ be the family of all nonempty finite subsets of $K'$, and for $A\in\mathfrak{F}$ let $\mathfrak{F}_A=\set{B\in\mathfrak{F}:A\subset B}$. Note that the sets $\mathfrak{F}$ and $\mathfrak{F}_A$ are directed by inclusion. We will now construct a net $(g_A)_{A\in\mathfrak{F}}$ in $\Lip_0(M)$ that satisfies conditions \ref{g_positive Lipschitz}--\ref{g_outside support} above with $\varepsilon=\varepsilon_{\abs{A}}$ (where $\abs{A}$ denotes the cardinality of $A$), and also these two:
\begin{enumerate}[itemsep=2pt,label={\upshape{(\roman*)}}]
\setcounter{enumi}{5}
\item\label{g_reasonably big at phi} $\abs{\duality{g_A,\phi}}\geq\frac{c}{8}$,
\item\label{g_decreasing} if $E\subset A$ then $g_A(x)\leq g_E(x)$ for every $x\in M$.
\end{enumerate}
This will be enough to end the proof. Indeed, $(g_A)_{A\in\mathfrak{F}}$ decreases pointwise to $0$ because $g_A(x)\leq 2\varepsilon_n$ whenever $\abs{A}\geq n$ and either $x\in A$ or $x\in M\setminus K'$ by \ref{g_max} and \ref{g_outside support} respectively, but $\abs{\duality{g_A,\phi}}\geq\frac{c}{8}$ for every $A\in\mathfrak{F}$, contradicting the normality of $\phi$.

We proceed by induction on $n=\abs{A}$. For $n=1$, i.e. singletons $A=\set{x}$ with $x\in K'$, let $g_A$ be the function $g$ given by Claim \ref{cl:existence_one_function} for $\varepsilon=\varepsilon_1$. It clearly satisfies \ref{g_positive Lipschitz}--\ref{g_reasonably big at phi}, and also \ref{g_decreasing} by vacuity. Now let $n>1$, assume that the functions $g_A$ have been constructed for all nonempty subsets $A\subset K'$ with fewer than $n$ elements, and fix $A\subset K'$ with $\abs{A}=n$. To complete the induction, it suffices to prove that there exists $g_A$ satisfying \ref{g_positive Lipschitz}--\ref{g_outside support} and \ref{g_reasonably big at phi}--\ref{g_decreasing} with $\varepsilon=\varepsilon_n$.

To this end, denote $h=\bigwedge_{E\subsetneq A} g_E$, which satisfies conditions \ref{g_positive Lipschitz}--\ref{g_outside support} with $A$ and $\varepsilon=\varepsilon_{n-1}$. Next, for any $B\in\mathfrak{F}_A$ let $\mathfrak{g}_B$ be the function given by Claim \ref{cl:existence_one_function} for the set $B$ and $\varepsilon=\varepsilon_n$. Notice that the function $\mathfrak{g}_B\wedge h$ satisfies conditions \ref{g_positive Lipschitz}--\ref{g_outside support} for $\varepsilon=\varepsilon_n$ and set $A$ because it is bounded by $\mathfrak{g}_B$, and also condition \ref{g_decreasing} because it is bounded by $h$. We will show that $g_A$ can be found among the functions $\mathfrak{g}_B\wedge h$, i.e. at least one of the functions $\mathfrak{g}_B\wedge h$ satisfies also condition \ref{g_reasonably big at phi}. The proof will proceed by contradiction, and we will need the following claim:

\begin{claim}
\label{cl:f_restriction_nonsep}
With the notation as above, if $\abs{\duality{\mathfrak{g}_B\wedge h,\phi}}<\frac{c}{8}$ for every $B\in\mathfrak{F}_A$, then there is a constant $\beta>0$ with the following property: for any $B\in\mathfrak{F}_A$, there exist $E\in\mathfrak{F}_B$ and $f\in\Lip_0(M)^+$ such that
\begin{enumerate}[itemsep=2pt,label={\upshape{(\alph*)}}]
\item \label{restriction_norm}$\lipnorm{f}\leq \beta$,
\item \label{restriction_support}$\supp(f)\subset \pare{\bigcup\limits_{x\in E}B(x,\varepsilon_n)}\setminus\pare{\bigcup\limits_{x\in B}B(x,\varepsilon_n)}$,
\item \label{restriction_duality}$\abs{\duality{f,\phi}}\geq \frac{c}{16}$.
\end{enumerate}
\end{claim}

\begin{proof}[Proof of Claim \ref{cl:f_restriction_nonsep}]
 Fix $B\in\mathfrak{F}_A$ and define $f=\wop{e}\pare{\mathfrak{g}_B-(\mathfrak{g}_B\wedge h)}$. Clearly $f\geq 0$ and $\lipnorm{f}\leq 2\alpha\pare{2+\frac{4}{r}R}$ by \eqref{eq:weighting_op_norm}. Suppose that $x\in B(b,\varepsilon_n)$ for some $b\in B$. If $x\notin K'$ then $e(x)=0$, and if $x\in K'$ then by \eqref{eq:epsilon_choice} we have
$$
\mathfrak{g}_B(x) \leq \mathfrak{g}_B(b)+\abs{\mathfrak{g}_B(x)-\mathfrak{g}_B(b)} \leq 2\varepsilon_n+\alpha\varepsilon_n < \varepsilon_{n-1}
$$
whereas $h(x)\geq\varepsilon_{n-1}$, so $\mathfrak{g}_B(x)\leq h(x)$. In any case $f(x)=0$ for all $x\in \bigcup_{b\in B}B(b,\varepsilon_n)$. Moreover,
$$
\abs{\duality{f,\phi}}=\abs{\duality{\mathfrak{g}_B-(\mathfrak{g}_B\wedge h),\phi}}\geq \abs{\duality{\mathfrak{g}_B,\phi}}-\abs{\duality{\mathfrak{g}_B\wedge h,\phi}} > \frac{c}{4}-\frac{c}{8} = \frac{c}{8} .
$$

Similarly to functions $e$ and $e^{\prime}$ introduced above, for a given $E\in\mathfrak{F}_B$ define the function
$$
e_E(x)=\pare{1-\frac{1}{\varepsilon_n}d(x,E)}\vee 0 \quad\textup{ for all }x\in M,
$$
which clearly satisfies that $\supp(e_E)\subset\bigcup_{x\in E}B(x,\varepsilon_n)$.
Then the net \mbox{$(\wop{e_E}(f))_{E\in\mathfrak{F}_B}$} is a norm-bounded increasing net in $\Lip_0(M)^+$ converging pointwise to $f$. Indeed, by \eqref{eq:weighting_op_norm} we have $\lipnorm{\wop{e_E}(f)}\leq \beta$, where
$$
\beta = \pare{1+\frac{1}{\varepsilon_n}\pare{R+\frac{r}{4}}}\cdot 2\alpha\pare{2+\frac{4}{r}R}
$$
does not depend on $B$ or $E$, and the rest is immediate from the definition. Hence the normality of $\phi$ implies that $\duality{\wop{e_E}(f),\phi}$ converges to $\duality{f,\phi}$, and in particular there exists $E\in\mathfrak{F}_B$ such that
$$
\abs{\duality{f,\phi}-\duality{\wop{e_E}(f),\phi}}<\frac{c}{16} .
$$
The function $\wop{e_E}(f)$ satisfies the requirements of the claim. Indeed, we have already verified \ref{restriction_norm}, \ref{restriction_support} follows from $\supp(\wop{e_E}(f))\subset \supp(f)\cap\supp(e_E)$, and we get \ref{restriction_duality} from
$$
\abs{\duality{\wop{e_E}(f),\phi}} \geq \abs{\duality{f,\phi}}-\abs{\duality{f,\phi}-\duality{\wop{e_E}(f),\phi}} > \frac{c}{8}-\frac{c}{16} = \frac{c}{16} .
$$
This ends the proof of Claim \ref{cl:f_restriction_nonsep}.
\end{proof}

To conclude our main argument, suppose that $\abs{\duality{\mathfrak{g}_B\wedge h,\phi}}<\frac{c}{8}$ for every $B\in\mathfrak{F}_A$. We then construct sequences $(B_n)\subset\mathfrak{F}_A$ and $(f_n)\subset\Lip_0(M)^+$ as follows: take $B_0=A$, and for any $n\in\NN$ let $B_n$ and $f_n$ be the set $E$ and function $f$, respectively, given by Claim \ref{cl:f_restriction_nonsep} for $B=B_{n-1}$. Then the sequence $(f_n)$ is norm-bounded by \ref{restriction_norm} and has pairwise disjoint supports by \ref{restriction_support}. However it is not weakly null due to \ref{restriction_duality}, which is in contradiction with Lemma \ref{lm:ccmgr}. This ends the proof of Theorem \ref{th:normal}.
\end{proof}

\section*{Acknowledgments}
R. J. Aliaga was partially supported by the Spanish Ministry of Economy, Industry and Competitiveness under Grant MTM2017-83262-C2-2-P. E. Perneck\'a was supported by the grant GA\v CR 18-00960Y of the Czech Science Foundation.

The authors would like to thank Michal Doucha and Richard J. Smith for their helpful comments and the anonymous referees for their nice suggestions.


\end{document}